\documentclass[11pt, twoside]{article} 
\usepackage{enumerate}	 %recommended package
\usepackage{mathtools}	 %recommended package
\usepackage{bm}		 %recommended package
\usepackage[sort,nocompress]{cite} %recommended package
\usepackage{enumitem}

\usepackage[width=136mm]{caption} %obligatory if the paper contains figures or tables

\usepackage{amsmath}    %obligatory package
\usepackage{amssymb}   	%obligatory package
\usepackage{amsthm}     %obligatory package
\usepackage{graphicx}    %obligatory package

\usepackage{textcomp}    %obligatory package
\usepackage[T1]{fontenc} %obligatory package
\usepackage{marvosym}  %obligatory package (for the \Letter symbol for indicating corresponding author)

\usepackage{authblk} 	%obligatory package
\usepackage[usenames]{xcolor} %obligatory package
\usepackage{lastpage} 	%obligatory package

\usepackage[latin1]{inputenc} % allows accented characters									
\usepackage{indentfirst}
\usepackage{amsfonts,amscd}	% symbols and special characters
\usepackage{empheq}
\usepackage{latexsym}		% special characters 
\usepackage{graphicx,graphics,epsfig}								
\usepackage{graphpap}										
\usepackage{float}
\usepackage{xcolor}
\usepackage{makeidx}      % remissive index
\usepackage{appendix}	  % appendix
\usepackage{multicol}										
\usepackage{color}

\usepackage[normalem]{ulem}
\usepackage{changepage}
\usepackage{thmtools, thm-restate}
\usepackage{hyperref} % optional arguments: [dvipdfm], [dvips]
\hypersetup{
	colorlinks=true,
	linkcolor=blue,         
   citecolor=blue,          
   urlcolor=blue          
}

\usepackage[dvips=false,pdftex=false,vtex=false]{geometry}
\usepackage{dsfont} % fonts for sets of nymbers, reals, complex, integers, etc.

\numberwithin{equation}{section}
\numberwithin{table}{section}
\numberwithin{figure}{section}

\newcommand{\orcid}[1]{\href{https://orcid.org/#1}{\includegraphics{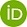}}} % orcid symbol with link to orchid page

\newtheorem{theorem}{Theorem}[section]

\newtheorem{Claim}[theorem]{Claim}

\theoremstyle{definition}

\theoremstyle{definition}

\theoremstyle{definition}

\newenvironment{claimproof}{%
  \begin{proof}[Proof of Claim~\theClaim]%
}{%
	\end{proof}%
}
\usepackage[capitalise, noabbrev]{cleveref}
\usepackage{subcaption}
\usepackage{tikz,ifthen}
\usetikzlibrary {arrows.meta,calc,positioning} 
\usepackage{float} %For appendix problem
\definecolor{MyDarkOrange}{RGB}{230,97,1}
\definecolor{MyLightOrange}{RGB}{253,184,99}
\definecolor{MyDarkPurple}{RGB}{178,171,210}
\definecolor{MyLightPurple}{RGB}{94,60,153}

\newcommand{\tw}{\operatorname{\mathsf{tw}}}
\newcommand{\problemdef}[3]{%
    \begin{center}
	\fbox { \parbox[c] {0.9\textwidth} {
    \textsc{{#1}}\\    
         \textbf{Input:} #2 \\
         \textbf{Goal:} #3}}
    \end{center}
}
\def\NP{\mathsf{NP}}
\def\P{\mathsf{P}}
\def\W2{\mathsf{W[2]}}
\newcommand{\tu}[1]{\widetilde{u}_{#1}}
\newcommand{\td}[1]{\widetilde{d}_{#1}}

\DeclareMathOperator*{\argmin}{arg\,min}
\def\final{0}  % set this to 1 to get a comment-free version
\def\iflong{\iffalse}
\ifnum\final=0  %namely if we allow comments in the output
\newcommand{\maxnote}[1]{{\color{red}[{\tiny \textbf{Max:} \bf #1}]\marginpar{\color{red}*}}}
\newcommand{\lmnote}[1]{{\color{purple}[{\tiny \textbf{Mirabel:} \bf #1}]\marginpar{\color{purple}*}}}
\newcommand{\pnote}[1]{{\color{blue}[{\tiny \textbf{Panni:} \bf #1}]\marginpar{\color{blue}*}}}
\newcommand{\dnote}[1]{{\color{olive}[{\tiny \textbf{Dani:} \bf #1}]\marginpar{\color{olive}*}}}
\newcommand{\linote}[1]{{\color{orange}[{\tiny \textbf{Lilla:} \bf #1}]\marginpar{\color{orange}*}}}
\else % in this case [final=1] we don't want any comments to show
\newcommand{\maxnote}[1]{}
\newcommand{\lmnote}[1]{}
\newcommand{\pnote}[1]{}
\newcommand{\dnote}[1]{}
\newcommand{\linote}[1]{}
\fi
\title{Diameter reduction via arc reversal}

\author[1]{\textbf{Panna Geh{\'e}r \orcid{0009-0000-9718-8318}}\footnote{This research was supported by the Lend{\"u}let Programme of the Hungarian Academy of Sciences -- grant number LP2021-1/2021; by the Hungarian National Research, Development and Innovation Office -- NKFIH, grant number FK128673; by the Ministry of Innovation and Technology of Hungary from the National Research, Development and Innovation Fund, financed under the ELTE TKP 2021-NKTA-62 funding scheme; by CONAHCyT: CBF 2023-2024-552.
E-mail: lyd21@student.elte.hu}} %Please write the email address of the corresponding author here
\author[2]{\textbf{Max K{\"o}lbl \orcid{0000-0002-5715-4508}}}
\author[1,3]{\textbf{Lydia Mirabel Mendoza-Cadena \orcid{0000-0002-4805-0029}}}
\author[1]{\textbf{Daniel P. Szabo \orcid{0009-0009-7263-1614}}}

\affil[1]{Department of Operations Research, E\"otv\"os Lor\'and University, Budapest, Hungary}
\affil[2]{Department of Pure and Applied Mathematics, Osaka University, Osaka, Japan}
\affil[3]{MTA-ELTE Matroid Optimization Research Group, E\"otv\"os Lor\'and University, Budapest, Hungary.}

\newcommand\shorttitle{Diameter reduction via arc reversal}

\makeatletter
\def\volume#1{\def\@volume{#1}}

\def\shortauthor#1{\def\@shortauthor{#1}}

\def\shorttitle#1{\def\@shorttitle{#1}}

\makeatother

\makeatletter
\let\thetitle\@title
\let\thedate\@date
\let\thevol\@volume
\let\theauthor\@author
\let\theshortauthor\@shortauthor
\let\theshorttitle\@shorttitle 
\makeatother

\makeatletter
\renewcommand{\maketitle}{\bgroup\setlength{\parindent}{0pt}

\parindent=1em
\renewcommand{\thefootnote}

\vspace{1truecm}
\begin{center}{\vbox{\titlefont\@title}}\end{center}
\vspace{0.5truecm}
\begin{center}{\@author} \end{center}

\egroup
}

\renewcommand{\thefootnote}{\fnsymbol{footnote}}
\renewcommand{\@fnsymbol}[1]{%
    \ifcase#1 \or {\,\Letter\!} \or\textasteriskcentered\or \textasteriskcentered\textasteriskcentered 
    \else\@ctrerr\fi}
\makeatother

\newcommand*{\titlefont}{\fontsize{18}{21.6}\selectfont\textbf}

\makeatletter
\renewcommand\@author{\ifx\AB@affillist\AB@empty\AB@author\else
      \ifnum\value{affil}>\value{Maxaffil}\def\rlap##1{##1}%
    \AB@authlist\\[\affilsep]\vbox{\AB@affillist}
    \else  \AB@authors\fi\fi}
\makeatother

\begin{document}

\maketitle
\thispagestyle{plain}
\renewcommand{\thefootnote}{\arabic{footnote}}   % don't delete!!!
\setcounter{footnote}{0}     % don't delete!!!
\setcounter{page}{1} % The editors will insert the correct initial pagenumber
\label{FirstPage}	

\pagestyle{myheadings} \markboth{  $\hspace{1.5cm}$  \hfil  P. Geh{\'e}r, M. K{\"o}lbl, L.M. Mendoza-Cadena and D.P. Szabo
\hfil $\hspace{3cm}$ } {\hfil$\hspace{1.5cm}$
{Diameter reduction via arc reversal}
\hfil}

%%%%%%%%%%%
% Minipage with: dedication, Abstract, keywords, and MSC
%%%%%%%%%%%
\begin{center}
\noindent
\begin{minipage}{0.85\textwidth}\parindent=15.5pt

% \smallskip
% \begin{center}
% \large{\textsl{Dedicated to Professor Renato Tribuzy\\ on the occasion of his 75th birthday}}
% \end{center}

\smallskip

{\small{
\noindent {\bf Abstract.} 
The diameter of a directed graph is the maximum distance between any pair of vertices. We study a problem that generalizes \textsc{Oriented Diameter}:  For a given directed graph and a positive integer $d$, what is the minimum number of arc reversals required to obtain a graph with diameter at most $d$?
% The diameter of a directed graph is the maximum distance among any pair of vertices. We study the following problem, which is a generalization of \textsc{Oriented Diameter}: For a given directed graph and a positive integer $d$, what is the minimum number of arc reversals required to obtain a graph with diameter at most $d$?
We investigate variants of this problem, considering the number of arc reversals and the target diameter as parameters. We show hardness results under certain parameter restrictions, and give polynomial time algorithms for planar and cactus graphs.
This work is partly motivated by the relation between oriented diameter and the volume of directed edge polytopes, which we show to be independent.
\smallskip

% Please enter at most 6 keywords here with lowercase letters separated by commas.
\noindent {\bf{Keywords:}}  diameter reduction, arc reversal, oriented diameter.
\smallskip

% Please enter at most 5 Mathematics Subject Classification codes here. Please use 2010 classification codes, which can be found on the following link: http://www.ams.org/msc//msc2010.html.
\noindent{\bf{2020 Mathematics Subject Classification:}} 05C85, 68R10
}
}
\end{minipage}
\end{center}

%%%%%%%%%%%%%
% end of minipage with abstract, dedication, keywords and MSC
%%%%%%%%%%%%%
% Main body of the article
%-------------------------------------------------------------
\section{Introduction\label{sec:intro}} % Please enter the title of your first section (only the first letter of the title should be capital).
Inspired by an application in road traffic control, to make traffic flow more efficient when using only one-way streets,  Robbins~\cite{robbins1939strongconnectivity} states that an undirected graph admits a strongly connected orientation, if and only if it is 2-edge-connected\footnote{A graph is $k$-edge-connected (resp. $k$-arc-connected) if at least $k$ edges (resp. arcs) have to be removed in order to destroy connectivity.}.
However, what makes an orientation ``good'' can vary according to the problem description. For example, we might want to minimize the distance from a given vertex to all other vertices or we might seek to minimize the average distance between all pairs of vertices. Here, we consider an \emph{oriented diameter problem} on directed graphs. In this problem our goal is to minimize the maximum distance between the vertices (called the \emph{diameter} of the graph) by changing the direction of as few arcs as possible. One motivation for this problem is the following: Suppose that there are $n$ airports, of which some ordered pairs are connected by a flight. For a given integer $d$ how many flights do we have to establish to make sure that from each airport one can get to another by changing planes at most $d$ times? 
Other instances where knowing the diameter of a graph is helpful include the analysis of social networks as well as parallel computing, see e.g. \cite{WattsStrogatzapplication1, Chungapplication2}.

To formalize the problem, we introduce some notation. Given a graph (directed or undirected) $D=(V, \, A)$, we denote an edge between $u$, $v \in V$ by $uv$ and an arc from $u$ to $v$ by $(u, \, v)$. We denote a path $P$ between $u, \, v \in V$ by $(u, \, w_1, \, w_2, \dots, w_\ell, \, v)$ where $w_0=u, \, w_1, \, \dots, \, w_\ell, \, w_{\ell+1}=v$ are the vertices  of the path so that $(w_i, \, w_{i+1}) \in A$ for all $i= 0, \dots,\ell$. The \emph{length} of $P$ is its number of edges (or arcs) and it is denoted by $|P|$.
For any $u, \, v \in V$, the \emph{distance} between them is the length (or weight if the edges are weighted) of the shortest path connecting $u$ and $v$; by convention, this distance is $\infty$ if there is no path connecting those vertices. The \emph{diameter} of the graph $D$, denoted by $\operatorname{diam}(D)$ is the maximum distance among its pairs of vertices.
An \emph{arc reversal} of $e=(u, \, v)$ in a directed graph $(V, \, A)$ results in a graph $(V, \, (A \setminus (u, \, v))\cup (v,u) )$. 

In this paper we consider the following problems.

\problemdef{Diameter Reduction In At Most $k$ Steps ($k$-Reversals)} 
    {A digraph $D=(V, \, A)$ with diameter $d_D$, and two integers $2 \leq d <d_D$, and $k\geq 0$.}
    {Decide whether a digraph of diameter at most $d$ can be obtained from $D$ by performing at most $k$ arc reversals.}

\problemdef{Weighted $k$-Reversals} 
    {A digraph $D=(V, \, A)$ with diameter $d_D$, nonnegative edge weights $w:A\to \mathbb{R}_{\geq 0}$, and two integers $2 \leq d <d_D$, and $k\geq 0$.}
    {Decide whether a digraph of diameter at most $d$ can be obtained from $D$ by performing arc reversals of weight at most $k$.}

\paragraph{Previous Results.} 
The \textsc{Oriented Diameter} problem takes an undirected simple graph and an integer \( d \geq 0 \), seeking an orientation with a diameter of at most \( d \). Chv{\'a}tal and Thomassen~\cite{chvatal1978distances} first addressed this problem, proving that finding an orientation with a diameter of 2 is \(\NP\)-hard, and extending their result to any diameter \( d \geq 4\). The case for \( d=3 \) remains unresolved. The hardness of the problem for planar graphs is also unknown, despite several attempts to explore it.
Bensmail, Duvignau, and Kirgizov~\cite{bensmail2016complexity} recently explored orientations in terms of the so-called {\it weak diameter}. A digraph $D$ has $d$-weak (resp. $d$-strong) diameter if the maximum weak (resp. strong) distance between any two vertices $u$ and $v$, i.e.\ the minimum (resp. maximum) of the distance from $u$ to $v$ and that from $v$ to $u$, is $d$. They proved that deciding if an undirected graph has a $d$-weak orientation is $\NP$-complete for $d \geq 2$ and conjectured that the same is true for $d$-strong orientations. This conjecture would complete the results of Chv{\'a}tal and Thomassen for $d = 3$.
Eggemann and Noble~\cite{eggemann2009minimizing, eggemann2012complexity} developed an FPT algorithm for \textsc{Oriented Diameter} in planar graphs, parameterized by treewidth. Mondal, Parthiban, and Rajasingh~\cite{mondal2022oriented} focused on triangular grid graphs, presenting a polynomial-time algorithm for this case, while also demonstrating that the weighted version of \textsc{Oriented Diameter} is weakly $\NP$-complete for planar graphs with bounded pathwidth. Additionally, Fomin, Matamala, and Rapaport~\cite{fomin2004complexity} examined chordal graphs, providing an approximation algorithm for \textsc{Oriented Diameter} and proving that determining if a chordal graph can achieve diameter $k$ is $\NP$-complete.
Ito et al.~\cite{ito2023monotone} studied the problem of finding a graph orientation that maximizes arc-connectivity\footnote{Arc-connectivity is the maximum integer $\lambda$ such that every non-empty subset $X \subsetneq V$ has at least $\lambda$ arcs leaving $X$.}. Hoppenot and Szigeti~\cite{hoppenot2024reversing} explored digraphs with $k$-arc-connectivity at most $\lfloor (k+1)/2\lfloor$, showing that if reversing a subset of arcs $F$ can achieve $k$-arc-connectivity, then reversing only one arc of $F$ does not reduce overall arc-connectivity. Additional problems related to arc-connectivity and reversals are discussed by Bang-Jensen, H{\"o}rsch, and Kriesell in~\cite{banJensen2023complexity}, while Bang-Jensen, Costa Ferreira da Silva and Havet \cite{BangJensen2022inversion} consider the minimum number of reversals needed to make a digraph acyclic.

\noindent\textbf{Our results.}\quad We summarize our results supplemented with the remaining (already solved or still open) cases in \cref{table:complexity}. Note that if \textsc{$k$-Reversals} is solvable on planar graphs, this implies that \textsc{Oriented Diameter} on planar graphs is also solvable (see \cref{sec:hardness}). On the other hand, if \textsc{Oriented Diameter} is $\NP$-hard for planar graphs, then \textsc{$k$-Reversals} is also $\NP$-hard on planar graphs. 

\begin{table}[h] \small
    \centering
    \caption{\small Complexity of the problems. Planar graphs is abbreviated as PG, while cactus graphs as CG. $\NP$-hard, Weakly $\NP$-hard, and $\W2$-hard are abbreviated by $\NP$h, W$\NP$h and $\W2$-h, respectively. The deferred result is marked with $\star$.}
    \label{table:complexity}
    \renewcommand*{\arraystretch}{0.85}
    \begin{tabular}{|c|c|c|c|} 
    \hline
    \textbf{Problem} & \textbf{Parameter} & \textbf{Fixed} & \textbf{Complexity} \\ \hline
         \textsc{$k$-Reversals}&  &  & $\NP$h -- Thm.~\ref{thm:hardness}\\
         \hline
         \textsc{$k$-Reversals}&  &  d& $\NP$h -- Thm.~\ref{thm:hardness}\\
         \hline
         \textsc{$k$-Reversals} &  &  k& $\P$ -- Thm.~\ref{thm:bruteforce} \\
         \hline
         \textsc{$k$-Reversals}&  k&  d& $\W2$h -- Thm.~\ref{thm:W2-hard_with_k}\\
         \hline
         \textsc{$k$-Reversals} (PG)&  &  & Open\\
         \hline
         \textsc{$k$-Reversals} (PG)&  d&  & FPT -- Sec.~\ref{sec:algo_eggeman_noble}\\
         \hline
         \textsc{$k$-Reversals} (PG)&  k&  & Open\\
         \hline
         \textsc{$k$-Reversals} (CG)&  &  & $\P$ -- Thm.~\ref{thm:cactus}\\
         \hline
         \textsc{Weighted $k$-Reversals} (CG)&  &  & W$\NP$h --   Thm.~\ref{thm:weaklyNP_wcactus}\\
         \hline
         \textsc{Oriented Diameter} (PG)&  &  & Open\\
         \hline
         \textsc{Oriented Diameter} (PG)&  d&  & FPT --\cite{eggemann2009minimizing}\\
         \hline
    \end{tabular}
\end{table}

We organize the remainder of the paper as follows. In \cref{sec:hardness}, we explore some hardness results for  \textsc{$k$-Reversals} with some parameters, while in \cref{sec:algorithm} we present polynomial algorithms for special cases. Additionally, the results in \cref{sec:polytopes} offer counter-examples to a conjecture regarding directed edge polytopes. We defer the preliminaries section and some of our results and proofs {to the appendix};  statements with deferred proofs are marked with the symbol~$\star$.

%-------------------------------------------------------------
\section{Hardness results} \label{sec:hardness}
We show that  \textsc{$k$-Reversals} is $\NP$-hard. First, recall that \textsc{Oriented Diameter} has as input a \emph{simple} undirected graph $G=(V, \, E)$ and $d \in \mathbb{Z}_{\geq 0}$, and the goal is  to find an orientation such that the diameter is at most $d$. Chv{\'a}tal and Thomassen \cite{chvatal1978distances} showed that it is  $\NP$-hard to decide whether an undirected graph admits an orientation of diameter 2. The authors also showed that for every $d \geq 4$, it is $\NP$-hard to determine if an undirected graph can be oriented with diameter $d$.

\begin{theorem}\label{thm:hardness}
    The problem \textsc{$k$-Reversals} is $\NP$-hard.
\end{theorem}
\begin{proof}
A solution to \textsc{$k$-Reversals} with parameters $k = |A|$ and $d$ also yields a solution to \textsc{Oriented Diameter} with parameter $d$.
By the above result of Chv{\'a}tal and Thomassen, \textsc{$k$-Reversals} is $\NP$-hard.
\end{proof}

For the rest of the section, we distinguish between weakly $\NP$-hard problems, which are hard only with binary input, and strongly $\NP$-hard problems, which remain hard with unary input (see e.g.~\cite{cygan2015parameterized}). For FPT algorithms, intractability often relies on reductions to complete problems in the $W[t]$ hierarchy.
%-------------------------------------------------------------
\subsection{W[2]-hardness for number of arc reversals \label{sec:W2hard_for_number_flips}}
\begin{restatable}[$\star$]{theorem}{thmWtwoHardByk}\label{thm:W2-hard_with_k}
\textsc{$k$-Reversals} is $\W2$-hard when parameterized by $k$.
\end{restatable}
\begin{proof}[Proof sketch.] 
    We reduce from \textsc{Dominating Set}, a $\W2$-complete problem (e.g.~\cite{cygan2015parameterized}). In this problem, we are given an undirected graph $G = (V,E)$ and an integer $\ell \geq 0$, the goal is to determine if there is a dominating set of size at most $\ell$. The reduction from an instance $(G,\ell)$ of \textsc{Dominating Set} is a digraph $H$ with diameter four. The generated instance $(H,3,\ell)$ has an arc reversal if and only if   $(G,\ell)$ has a dominating set.
\end{proof}

%-------------------------------------------------------------
\subsection{Weak NP-hardness for weighted cactus graphs}
We consider \textsc{$k$-Reversals} for cactus graphs which is weakly $\NP$-hard. Later, in Section~\ref{section: cactus} we show a dynamic programming algorithm.  Recall that cactus graphs are connected graphs where any two simple cycles share at most one vertex. We denote $[n]=\{1, 2, \dots, n \}$ for an integer $n\geq 1$.
\begin{theorem} \label{thm:weaklyNP_wcactus}
The problem \textsc{Weighted $k$-Reversals} is weakly $\NP$-hard for cactus graphs.
\end{theorem}
\begin{proof}
    We make a reduction from \textsc{Partition} which is known to be weakly $\NP$-hard. Recall that in \textsc{Partition}, we are given
        {$n$ positive integers $A = \{a_1, a_2, \dots,a_n\}$ with sum $2b$}, and the goal is to
        {determine if there is a subset $S \subseteq [n]$ such that
    $\sum_{i \in S} a_i= \sum_{i \not\in S} a_i=b$.}

    Given an input $(A, b)$ for \textsc{Partition}, we construct a weighted cactus graph formed by two paths on the same vertex set-- one weighted by $A$ and the other with unit weights.
    Let $G=(V, \, E)$ be a graph with vertices $v_1, \, v_2,\ldots, v_{n+1}$. For all $i\in [n]$ we have two arcs from $v_i$ to $v_{i+1}$ denoted by $e_i$ and $f_i$.  
    Set the weights $w(e_i)=1$ and $w(f_i)=(a_i+1)$.
    Finally, we set $k = b+n$ and $d = (b+n)$.
    This is our input for \textsc{Weighted $k$-Reversals} on a  cactus graph, as shown in \cref{fig:weighted_cactus}.
    Note that each arc is oriented either clockwise (for all arcs $e_i$) or counterclockwise (for all arcs $f_i$). This graph has diameter $\operatorname{diam}(G) = \infty$.

    \begin{figure}[t]
        \centering
        \begin{tikzpicture}[myNode/.style={circle, draw,  minimum size=3em},
        myArrow/.style={draw,line width = 1pt, -{Stealth[length=3mm]}}]
            \foreach \i \labi in {1/{$v_1$}, 2/$v_2$, 4/{$v_n$}, 5/{$v_{n+1}$}}{%
                \node[myNode] (v\i) at (2*\i, 0 ) {\labi};    
            };
            \node (v3) at (2*3, 0) { }; 
            \node (v6) at ($(v3.east)!0.5!(v4.west)$) {\Large $\cdots$}; 
            \foreach \a \b \laba in {1/2/1,2/3/2,4/5/n}{%
                \draw[myArrow, bend left] (v\a) to node[above, ] {1} (v\b);
                \draw[myArrow, bend right] (v\a) to node[below] {$a_{\laba} + 1$} (v\b);}   
        \end{tikzpicture}
        \caption{Instance of  \textsc{Weighted $k$-Reversals} on cactus graphs generated by an instance of \textsc{Partition}.}
        \label{fig:weighted_cactus}
    \end{figure}
    
   If there is an orientation with diameter $(b+n)$, the distance of $v_1$ and $v_{n+1}$ are such that the arcs directed clockwise  sum to at most $(b+n)$, and the arcs directed counterclockwise sum to at most $(b+n)$.
   As the total weight of the arcs is $2(b + n)$, it follows that both the distance from $v_1$ to $v_n$ and the distance from $v_n$ to $v_1$ must each be exactly $b + n$. If $S\subseteq [n]$ is the set of indices $i$ such that arc $f_i$ is directed clockwise, we have
    \vspace{-1ex}
    \begin{align*}
       \sum_{i\in S} (a_i+1) + \sum_{i \in A \setminus S} 1 &= \sum_{i \in A \setminus S} (a_i+1) + \sum_{i\in S} 1\\
       \sum_{i\in S} a_i + n &= \sum_{i \in A \setminus S} a_i + n\\
       \sum_{i\in S} a_i &= \sum_{i \in A \setminus S} a_i = b.
    \end{align*}
    Similarly, if $S\subseteq [n]$ is a set of indices with $\sum_{i \in S}a_i=b$,
    then reverse the arcs $e_i$ if $i \in S$ and reverse $f_i$ if $i \in A \setminus S$. We made $2n$ arc-reversals of weight $\sum_{i\in S} 1 + \sum_{i\in A \setminus S} (a_i+1) = (b+n)$. For all $i\in [n]$, $v_{i}$ and $v_{i+1}$ have two arcs in both directions. Thus, the unique path from $v_1$ to $v_2$ uses arcs $e_i$ if $i \in A \setminus S$ and arcs $f_i$ if $i \in S$, which has weight $b + n$. The rest of the arcs form a path from $v_n$ to $v_1$ with weight $b+n$, implying $\operatorname{diam}(G) = b+n$ as required.    
\end{proof}

%-------------------------------------------------------------
\section{Algorithms} \label{sec:algorithm}
In this section we give positive results for some special cases of the problem \textsc{$k$-Reversals}. Our first observation is that in case $k$ is fixed, examining all sets of size at most $k$ yields a polynomial time algorithm.

\begin{restatable}[$\star$]{theorem}{thmBruteForce} \label{thm:bruteforce}
    \textsc{$k$-Reversals} is polynomially solvable if $k$ is fixed.
\end{restatable}

%-------------------------------------------------------------
\subsection{Dynamic programming algorithm for cactus graphs }\label{section: cactus}

First, note that the hardness result for weighted cactus graphs does not apply here as it only showed weak $\NP$-hardness.
For unweighted cactus graphs we can construct a unary weighted cactus graph by contracting paths and weighting them by their lengths.
The key difference is that the orientations of a cactus graph are simpler to handle, as each cycle must be oriented either clockwise or counterclockwise. 

\begin{restatable}[$\star$]{theorem}{thmCactusDynamicAlg} \label{thm:cactus}
    \textsc{$k$-Reversals}  is solvable in polynomial time for unweighted cactus graphs.
\end{restatable}

%-------------------------------------------------------------
\subsection{FPT algorithm on diameter for planar graphs\label{sec:algo_eggeman_noble}}
\textsc{$k$-Reversals} problem can be solved in polynomial time on planar graphs using the results of Eggeman and Noble~\cite[Thm. 2.5]{eggemann2009minimizing}, by adding a extra factor of $k$, as shown in last section. {See the details in the appendix~\ref{section_deferred_FPT_k_fixed}.}

%-------------------------------------------------------------
\subsection{Directed edge polytopes \label{sec:polytopes}}
%-------------------------------------------------------------
A simple digraph $D=(V, \, A)$ is \emph{symmetric} if for every arc $(u, \, v) \in A$ its reverse $(v,u)$ is also contained in $A$. Given a symmetric digraph $D=(V, \, A)$ its \emph{symmetric edge polytope} is defined as the convex hull of the vectors $e_v-e_w\in\mathbb{R}^V$ where $(w, \, v)\in A$.
In recent years, symmetric edge polytopes have become a popular object of study.
It is natural to extend the definition to general directed graphs, leading to the notion of a \emph{directed edge polytope}, see e.g.\ \cite{Higashitani_edge_poly}. 
Studies of these objects are still limited, but observations suggested an inverse relationship between the diameter of a directed graph and the volume of the associated directed edge polytope.
Cactus graphs provide a family of counter-examples to this claim.

\begin{restatable}[$\star$]{theorem}{thmEdgePolytope}
\label{edge_polytope}
     The diameter of the graph does not depend on the volume of its associated directed edge polytope.
\end{restatable}

%-------------------------------------------------------------

\iffalse
\subsection{Comments}
For planar graphs, if $n \geq 120$, the diameter must be at least 3:

\begin{theorem} [Katona -- Szemer{\'e}di \cite{katona1967problem}]
    Let $D(V,E)$ be a directed graph with diameter 2. Then
    $$|E| \geq \frac{n}{2} \log_2\left({\frac{n}{2}}\right).$$

Also, if $|E| \geq n \{\log(n) \}$ then it can be orientated such that \it{diam}($G$) is 2.
\end{theorem}
\fi

\section*{Acknowledgements}
This work began at the 14th Eml{\'e}kt{\'a}bla Workshop. We thank Lilla T{\'o}thm{\'e}r{\'e}sz for introducing the problem and for her insightful discussions, as well as Florian H{\"o}rsch for his valuable contributions.

We thank the anonymous reviewers for their helpful feedback.
%-------------------------------------------------------------
\small
\bibliography{Diameter_reduction_via_arc_reversal}
\bibliographystyle{plain}

%-------------------------------------------------------------
% \clearpage

\newpage
\begin{appendices}

%-------------------------------------------------------------

\section{Preliminaries}\label{sec:prelim}
For a given positive integer $n$, we denote by $[n]$ the set $\{ 1, \, 2, \dots, n \}$.  The set of integer, and non-negative integers, non-negative reals, are denoted by $\mathbb{Z}$, $\mathbb{Z}_{\geq 0}$, $\mathbb{R}_{\geq 0}$, respectively.
\vspace{-2ex}
\paragraph{Graph Theory.} 
Given an undirected graph  $G= (V, \, E)$, a subset of vertices $S \subseteq V$ is called a \emph{dom\-i\-nat\-ing set} if for each vertex $v$ of $G$ either $v\in S$, or $v$ has a neighbor $u$ contained in $S$, that is, $uv \in E$ and $u \in S$.
\vspace{-2ex}
\paragraph{Special graph classes.}
When searching for polynomial time algorithms for $\NP$-hard problems, focusing on specific graph families can be effective. Planarity often simplifies these problems; a graph is planar if it can be drawn in the plane without crossings. Cactus graphs, a notable subclass of planar graphs, are connected graphs where any two simple cycles share at most one vertex. They can be represented as a tree, with nodes for cycles and edges for shared vertices.
\vspace{-2ex}
\paragraph{Complexity Theory.} In studying $\NP$-hard problems, the class $\NP$ alone isn't enough for intractability results. We distinguish between weakly $\NP$-hard problems, which are hard with binary input, and strongly $\NP$-hard problems, which remain hard with unary input. For FPT algorithms, intractability often relies on reductions to complete problems in the $W[t]$ hierarchy. For instance, \textsc{Clique} is $W[1]$-complete when parameterized by clique size, and \textsc{Dominating Set} is $W[2]$-complete by dominating set size; thus, they are typically assumed not to be FPT.
\vspace{-2ex}
\paragraph{Treewidth.}
The \emph{treewidth} of a graph measures its similarity to a tree. First defined by Halin and later rediscovered by Robertson and Seymour, it involves a \emph{tree decomposition} \((T, \mathcal{B})\), where \(T\) is a tree and \(\mathcal{B}\) is a collection of bags \(B_v \subseteq V(G)\) for each node \(v\) in \(T\). Each edge of \(G\) must have a bag containing both endpoints, and bags containing a vertex \(v\) must form a subtree of \(T\). The \emph{width} of \((T, \mathcal{B})\) is the largest bag size minus one, and the \emph{treewidth} \(\tw(G)\) is the minimum width across all decompositions. While computing treewidth is \(\NP\)-hard, it can be done in FPT time for graphs with bounded treewidth.

\begin{theorem} [Bodlaender \cite{bodlaender1996treewidth}] \label{thm_treewidth}
For a given integer $k$, there is a polynomial time algorithm that decides if a graph's treewidth is at most $k$ and finds a corresponding tree decomposition if it is.
\end{theorem}

%-----------------------------------
\section{Proofs for Section~\ref{sec:hardness}}
\thmWtwoHardByk*
\begin{proof}
   We make a reduction from \textsc{Dominating Set} which is known to be a $\W2$-complete problem (see e.g.\ Cygan et al.~\cite{cygan2015parameterized}). Recall that in \textsc{Dominating Set} we are given an undirected graph $G = (V,E)$, and $\ell \in \mathbb{Z}_{\geq 0}$, and the goal is to decide if there exists a dominating set of size at most $\ell$.
   
   Let $(G,\ell)$ be an instance of \textsc{Dominating Set}, let $V(G) = \{ v_1, \dots, v_n \}$. Using $(G,\ell)$ we construct an instance of \textsc{$k$-Reversals} as follows. 
    For each vertex $v_i$, we add the vertex-gadget $H_i$ (see a figure \cref{fig:vertex_gadget}) that contains the vertices:  $u_{i,1}, u_{i,2}$, referred as the \textit{upper vertices}; $d_{i,1}, d_{i,2}$, referred as the \textit{down vertices}; $\tu{i,1}, \tu{i,2}$, referred as the \textit{auxiliary upper vertices}; and  $\td{i,1}, \td{i,2}$, referred as the \textit{auxiliary down vertices.}
    We connect those vertices with the following arcs: For the upper and down vertices, add arcs $(u_{i,1}, u_{i,2})$, $(d_{i,1}, d_{i,2})$, $(u_{i,1}, d_{i,1})$, and $(d_{i,2}, u_{i,2})$. Connect each vertex with its auxiliary pair as $(u_{i,j},\tu{i,j}), (\tu{i,j}, u_{i,j})$ and $(d_{i,j},\td{i,j}), (\td{i,j},d_{i,j})$ for $j =1,2$. Connect all auxiliary vertices $x,y$ with arcs $(x,y),(y,x)$ for all pairs, except for pair $\td{i,1}$ and $\td{i,2}$.
    
   Now, we connect gadgets $H_i$ and $H_j$ in a symmetric way, meaning that the resulting arcs are the same if we change the role of $j$ and $i$ in the following description (see \cref{fig:gadgets_connected}).
   For each $v_iv_j \in E(G)$, and $H_i$, $H_j$ the gadgets for vertices $i$ and $j$, respectively, add arcs $(u_{i,1}, \, d_{j,1})$ and $(d_{j,2}, \, u_{i,2})$, and $(u_{j,1}, \, d_{i,1})$ and $(d_{i,2}, \, u_{j,2})$; note that the first two arcs are symmetrical to the last two.
   For each $v_i, v_j \in V(G)$, and $H_i$, $H_j$ the gadgets for vertices $i$ and $j$, respectively, connect all auxiliary vertices between $H_i$ and $H_j$ with arcs $(x,y), \, (y,x)$ for all pairs with $x \in H_i$ and $y\in H_j$; note that pair $\td{i,1}$ and $\td{j,2}$ can be connected.

\begin{figure}
    \centering
    \begin{subfigure}{0.4\textwidth}
    \centering
    \resizebox{\textwidth}{!}{
        \begin{tikzpicture}[myNode/.style={circle, draw,  minimum size=2.5em},myArrow/.style={draw,line width = 1pt, -{Stealth[length=3mm]}},
        myArrowDouble/.style={draw,line width = 1pt, {Stealth[length=3mm]}-{Stealth[length=3mm]}}]
            \coordinate (m_1) at (0, 7);
             \foreach \x in {i}{
                \node[myNode] (u\x1) at (m_1) {$u_{\x,1}$};
                \node[myNode, right = 3em of u\x1] (u\x2) {$u_{\x,2}$};
                \node[myNode, below = 2em of u\x1] (d\x1) {$d_{\x,1}$};
                \node[myNode, right = 3em of d\x1] (d\x2) {$d_{\x,2}$};
                \node[myNode, above left = 4pt and 12pt of u\x1] (U\x1) {$\tu{\x,1}$};
                \node[myNode, above right = 4pt and 12pt of u\x2] (U\x2) {$\tu{\x,2}$};
                \node[myNode, below left = 4pt and 12pt of d\x1] (D\x1) {$\td{\x,1}$};
                \node[myNode, below right = 4pt and 12pt of d\x2] (D\x2) {$\td{\x,2}$};
            
                \path [myArrow](u\x1) edge node[left] {} (u\x2); 
                \path [myArrow](u\x1) edge node[left] {} (d\x1);
                \path [myArrow](d\x1) edge node[left] {} (d\x2);
                \path [myArrow](d\x2) edge node[left] {} (u\x2);
                \foreach \a \b in {u\x1/U\x1,U\x1/u\x1,u\x2/U\x2,U\x2/u\x2,d\x1/D\x1,D\x1/d\x1,d\x2/D\x2,D\x2/d\x2}{
                    \draw (\a) edge [myArrow, bend right] (\b);};
                \foreach \a \b in {U\x1/U\x2, D\x1/U\x1, U\x2/D\x2}{
                    \draw (\a) edge [myArrowDouble, bend left] (\b);};
                \coordinate (aux1) at ($(U\x2)+(0.5,0.5)$);
                \draw[myArrowDouble] (U\x1) to [out=60, in=130]  (aux1) to [out=130-180,in=20] (D\x2);
                \coordinate (aux2) at ($(U\x1)+(-0.5,0.5)$);         
                \draw[myArrowDouble] (U\x2) to [out=130, in=130-90]  (aux2) to [out=130-90-180,in=160] (D\x1);}
    \end{tikzpicture} }
    \caption{Vertex gadget.\label{fig:vertex_gadget}}
    \end{subfigure}\hfill    
    \begin{subfigure}[b]{0.5\textwidth}
    \resizebox{\textwidth}{!}{
        \begin{tikzpicture}[myNode/.style={circle, draw,  minimum size=2.5em},myArrow/.style={draw,line width = 1pt, -{Stealth[length=3mm]}}]
            \coordinate (m_1) at (0, 7);
             \coordinate (m_2) at (4, 7);
             \foreach \x[count=\cx ] in {i,j}{
                \node[myNode] (u\x1) at (m_\cx) {$u_{\x,1}$};
                \node[myNode, right = 2em of u\x1] (u\x2) {$u_{\x,2}$};
                \node[myNode, below = 2em of u\x1] (d\x1) {$d_{\x,1}$};
                \node[myNode, right = 2em of d\x1] (d\x2) {$d_{\x,2}$};
                \path [myArrow](u\x1) edge node[left] {} (u\x2); 
                \path [myArrow](u\x1) edge node[left] {} (d\x1);
                \path [myArrow](d\x1) edge node[left] {} (d\x2);
                \path [myArrow](d\x2) edge node[left] {} (u\x2);}
    
            % Interconnecting edges
            \draw(ui1) edge[myArrow] (dj1);
            \draw (dj2.north west) edge[myArrow]  (ui2);
            \draw (uj1)  edge[myArrow] (di1.north east);
            \draw (di2.east) edge[myArrow] (uj2.south west);
    \end{tikzpicture}
    }
    \caption{The connectivity of two gadgets. Auxiliary nodes are omitted for simplicity. \label{fig:gadgets_connected}}
    \end{subfigure}

    \begin{subfigure}{\textwidth}
    \resizebox{\textwidth}{!}{
        \begin{tikzpicture}[myNode/.style={circle, draw,  minimum size=3em},
        myArrow/.style={draw,line width = 1pt, -{Stealth[length=3mm]}},
         myArrowDouble/.style={draw,line width = 1pt, {Stealth[length=3mm]}-{Stealth[length=3mm]}}]
             \coordinate (m_1) at (0, 6);
             \coordinate (m_2) at (8, 6);
             \coordinate (m_3) at (4, 0);
             \foreach \x in {1,2,3}{
                \node[myNode] (u\x1) at (m_\x) {$u_{\x,1}$};
                \node[myNode, right = 3em of u\x1] (u\x2) {$u_{\x,2}$};
                \node[myNode, below = 2em of u\x1] (d\x1) {$d_{\x,1}$};
                \node[myNode, right = 3em of d\x1] (d\x2) {$d_{\x,2}$};
                \node[myNode, above left = 2pt and 12pt of u\x1] (U\x1) {$\tu{\x,1}$};
                \node[myNode, above right = 2pt and 12pt of u\x2] (U\x2) {$\tu{\x,2}$};
                \node[myNode, below left = 2pt and 12pt of d\x1] (D\x1) {$\td{\x,1}$};
                \node[myNode, below right = 2pt and 12pt of d\x2] (D\x2) {$\td{\x,2}$};
            
                \ifthenelse{\x=1}
                    {\path [myArrow,line width = 2.5pt,MyLightOrange ](u\x1) edge node[left] {} (u\x2);}
                    {\path [myArrow](u\x1) edge node[left] {} (u\x2);} 
                \path [myArrow](u\x1) edge node[left] {} (d\x1);
                \path [myArrow](d\x1) edge node[left] {} (d\x2);
                \path [myArrow](d\x2) edge node[left] {} (u\x2);
                \foreach \a \b in {u\x1/U\x1,U\x1/u\x1,u\x2/U\x2,U\x2/u\x2,d\x1/D\x1,D\x1/d\x1,d\x2/D\x2,D\x2/d\x2}{
                    \draw (\a) edge [myArrow, bend right = 10] (\b);};
                \foreach \a \b in {U\x1/U\x2, D\x1/U\x1, U\x2/D\x2}{
                    \draw (\a) edge [myArrowDouble, bend left] (\b);};
                \coordinate (aux1) at ($(U\x2)+(0.5,0.5)$);
                \draw[myArrowDouble] (U\x1) to [out=60, in=130]  (aux1) to [out=130-180,in=20] (D\x2);
                \coordinate (aux2) at ($(U\x1)+(-0.5,0.5)$);         
                \draw[myArrowDouble] (U\x2) to [out=130, in=130-90]  (aux2) to [out=130-90-180,in=160] (D\x1);}
             %Interconnecting edges
            \foreach \i \j \b \bb \bbb \bbbb in {1/2/0/0/0/0, 2/3/-10/8/35/35, 3/1/40/40/10/-10}{
                \draw(u\i1) edge[myArrow, bend left =\b ] (d\j1);
                \draw (d\j2) edge[myArrow, bend right =\bb]  (u\i2);                
                \draw (u\j1)  edge[myArrow, bend right = \bbb] (d\i1);
                \draw (d\i2) edge[myArrow, bend left = \bbbb] (u\j2);}    
        \end{tikzpicture}
        }
        \caption{Example of reduction used on \cref{thm:W2-hard_with_k}. Directed graph $H$ constructed using the complete graph $G=(V,E)$ with $V=\{v_1,v_2,v_3\}$, where arcs between auxiliary nodes in distinct vertex gadgets are omitted. A solution to instance $(G,1)$ is $S = \{ v_1\}$, while a solution to \textsc{$k$-Reversals} is $F = \{ (u_{1,1}, u_{1,2})\}$ which is shown with a thick orange line. \label{fig:example_reduction}}
    \end{subfigure}
    \caption{Construction of the instance of \textsc{$k$-Reversals} defined in \cref{thm:W2-hard_with_k}. \label{fig:all_Figures}}    
\end{figure}

We call the resulting graph $H$. An example of the construction can be seen in \cref{fig:example_reduction}. To conclude the reduction, set $k = \ell$.
The gadgets have the following properties, which can be easily verified by checking the shortest paths among all the vertices. 

\begin{Claim}
\label{claim:Diam_Hi_Is_Four}
    For all $i=1, \, \dots, n$ it holds that $\operatorname{diam}(H_i) = 4$. Furthermore, the diameter of $H$ is exactly $4$.
\end{Claim}
\begin{claimproof}
    To show the first statement, first we prove that the shortest path between any pair of vertices has length at most $4$. For auxiliary vertices the statement holds trivially. Next, we show that any auxiliary vertex $a \in \{ \tu{i,1}, \, \tu{i,2}, \, \td{i,1}, \, \td{i,2}\}$ is at distance at most 3 from any upper and down vertex $y$. First consider pair  $a=\td{i,1}$, $y=d_{i,2}$ and pair $a=\td{i,2}$, $y=d_{i,1}$ with their paths
    \begin{align}
        &(\td{i,1},d_{i,1}, d_{i,2}), \label{eq:path_inside_gadget_aux1}\\
        &(\td{i,2}, \tu{i,1} , \td{i,1},d_{i,1}). \label{eq:path_inside_gadget_aux2}
    \end{align}
    For the rest, let $y \in \{ u_{i,1}, \, u_{i,2}, \, d_{i,1}, \, d_{i,2}\}$, and let $a_y$ be its auxiliary vertex. Then the path 
     \begin{equation}\label{eq:path_inside_gadget_allAux}
         (a,a_y,y)
     \end{equation}
     has length 2, as desired.
    
    Similarly, if we start from $x \in \{ u_{i,1}, \, u_{i,2}, \, d_{i,1}, \, d_{i,2} \}$ and we consider $y \in \{ u_{i,1},u_{i,2},d_{i,1},d_{i,2} \}$ (except for pair $x =d_{i,2}$ and $y=d_{i,1}$), we have the path    
    \begin{equation}
        (x,a_x,a_y,y) \label{eq:path_inside_gadget}
    \end{equation}
    
    which have length 3,  where $a_x$ and $a_y$ are the auxiliary vertices of $x$ and $y$ respectively. Finally, for $x =d_{i,2}$ and $y=d_{i,1}$, we have the path 
    \begin{equation}
        (d_{i,2},u_{i,2},\tu{i,2},\td{i,1},d_{i,1}) \label{eq:path_inside_gadget_length4}
    \end{equation}    
    with length 4. Note that this (not unique) path gives the diameter of $H_i$ yielding $\operatorname{diam}(H_i) = 4$.

    Now we prove that the diameter of $H$ is exactly $4$. 

    By the first part of the proof, the path from $d_{i,2}$ to $d_{i,1}$ has length 4 for any $i = 1, \dots, n$ (see Equation~\ref{eq:path_inside_gadget_length4}), and it is easy to see that using (auxiliary) vertices of other gadget do not decrease this value. 
    Similarly, for the rest of vertices $x,y \in V(H_i)$ for $i= 1, \dots, n$. 
    Hence, by the first part of the proof we only need to consider the distances between vertices of different gadgets.

    Any pair of vertices in $\{ \tu{i,1}, \, \tu{i,2}, \, \td{i,1}, \, \td{i,2}, \, \tu{j,1}, \, \tu{j,2}, \, \td{j,1}, \, \td{j,2} \}$ is at distance at most 2 by construction. 
     Now, let $H_i$ and $H_j$ be two different gadgets, and consider the path from $x_i \in \{ u_{i,1}, \, u_{i,2}, \, d_{i,1}, \, d_{i,2}\}$ to $y_j \in \{ u_{j,1}, \,  u_{j,2}, \, d_{j,1}, \, d_{j,2}\}$. Let $a_{i}$ and $a_j$ be the auxiliary vertices of $x_i$ and $y_j$, respectively. Hence, the path
    \begin{equation}
        (x_i, a_i, a_j, y_j) \label{eq:path_between_gadgets}
    \end{equation}
    has length 3, concluding the proof. 
\end{claimproof}

A solution of the \textsc{$k$-Reversals} instance can have an arc reversals from one type of arcs.
\begin{Claim}
\label{claim:characteristic_solution_flippling}
    Any solution to instance $(H, 3, \ell)$ consists of a set of arc reversal of the form $(u_{i,1}, \,  u_{i,2})$ for some indexes $i$.
\end{Claim}
\begin{claimproof}
    We check that reversing other arcs will require more than $\ell$ arc reversals.
    Notice that paths of length 4 are from $d_{i,2}$ to $d_{i,1}$ (see \cref{eq:path_inside_gadget_length4}). Thus we can focus only on reducing the distance from $d_{i,2}$ to $d_{i,1}$. The main idea is to consider an undirected shortest paths of length smaller than 4, and give the correct direction with as few reversals as possible.

    Note that changing the direction of arc $(d_{i,1}, \, d_{i,2})$ only affects gadget $H_i$, and using such type of arcs force us to have $n$ arc reversals. Also, it is easy to see that there is no undirected path of length two from $d_{i,2}$ to $d_{i,1}$.
    
    The undirected shortest paths between $d_{i,2}$ and $d_{i,1}$ with length 3 are:
    \begin{align}
        &(d_{i,2},u_{i,2}, u_{i,1}, d_{i,1}) \label{eq:undirected_path_between_downVertices1}\\
        &(d_{i,2},u_{j,2}, u_{j,1}, d_{i,1}) \label{eq:undirected_path_between_downVertices2} \text{ for some index $j$}.
    \end{align}
    It is enough to direct correctly at least one of those paths. Clearly, reversing the arc $(u_{i,1}, \, u_{i,2})$ is enough to modify such paths for the two connected gadgets $H_i$ and $H_j$, and any other modification requires at least two arcs.

    It only remains to prove that reversing the arc $(u_{i,1}, \, u_{i,2})$ affects gadgets $H_i$ and all the other gadgets connected to it. Indeed, this follows by paths on \cref{eq:undirected_path_between_downVertices1} and \cref{eq:undirected_path_between_downVertices2}.
\end{claimproof}

    The reduction is indeed correct, as it follows by last claims.
\begin{Claim}
\label{claim:correct_reduction}
   $(G,\ell)$ is a \texttt{Yes}-instance of \textsc{Dominating Set} if and only if $(H, \, 3, \, k =\ell)$ is a \texttt{Yes}-instance of \textsc{$k$-Reversals}.
\end{Claim}
\begin{claimproof}
    Let $S^* = (s_1, \dots, s_\ell)$ be a dominating set of $\ell$ elements. Consider the following $\ell$ arc reversals $F = \{ (u_{i,1}, \, u_{i,2}) \, \vert \, v_i \in S^* \}$. Let $H'$ be the graph obtained by reversing the arcs in $F$. We show that $\operatorname{diam} (H') = 3$. 

    First, note that auxiliary vertices are not modified by $F$ and the distance between any pair is still at most 3. 
    For all gadgets $H_i$, by paths of the form of $(x,a_x,a_y,y)$, we do not increase the distance from $x \in \{ u_{i,1}, \, u_{i,2}, \, d_{i,1}, \, d_{i,2} \}$ to  $y \in \{ u_{i,1}, \, u_{i,2}, \, d_{i,1}, \, d_{i,2} \}$ (except for pair $x =d_{i,2}$ and $y=d_{i,1}$)  where $a_x$ and $a_y$ are the auxiliary vertices of $x$ and $y$ respectively. 
    Similarly, by $(x_i, a_i, a_j, y_j)$ we do not increase any distance from $x_i \in \{ u_{i,1}, \, u_{i,2}, \, d_{i,1}, \, d_{i,2}\}$ to $y_j \in \{ u_{j,1}, \, u_{j,2}, \, d_{j,1}, \, d_{j,2}\}$. Also, all paths already mentioned have length 3.

    We only need to verify the distance from $d_{i,2}$ to $d_{i,1}$ for all $i \in \{ 1, \dots, n\}$. If $i$ is an index so that $v_i \in S^*$, then the path $(d_{i,2}, \, u_{i,1}, \, u_{i,2}, \, d_{i,1})$ has length 3 as desired. Otherwise, let $j$ be the index so that $v_i$ is adjacent to $v_j$ in $G$ and $v_j\in S^*$. Such index $j$ exists as $S^*$ is a dominating set. Thus, the path $(d_{i,2}, \, u_{j,2}, \, u_{j,1}, d_{i,1})$ has length 3, as desired. 

    To see the other direction, let $F^*$ be the set of $\ell$ arc reversals, and $H'$ the graph of diameter 3 obtained after the reversing. By Claim~\ref{claim:characteristic_solution_flippling}, we know that $F^*$ consists of arcs $(u_{i,1}, \, u_{i,2})$ for some gadgets $H_i$. Let us define $S = \{ v_i \,\vert \, (u_{i,1}, u_{i,2}) \in F^* \}$. Note that $|S| = \ell$. We show that $S$ is a dominating set. 
    Suppose there exists an index $i$ so that $v_i \notin S$ and $v_i$ has no adjacent vertex $v_j \in S$. This implies that distance from $d_{i,2}$ to $d_{i,1}$ is still 4, but this contradicts the optimality of $F^*$. 
\end{claimproof}           
\end{proof}

%------------------------------------------------------------
\section{Proofs for Section \ref{sec:algorithm}}

%------------------------------------------------------------
\thmBruteForce*
\begin{proof}
    Let $m$ be the number of arcs of the given graph.
    For all $i=1, \dots, k$ and $j=1, \dots, \binom{m}{i}$ consider all possible reversal sets $F_{ij}$. For each $F_{ij}$, using breadth first search we can calculate the diameter of the graph where the edge set $F_{ij}$ is reversed. Hence, we can decide in polynomial time whether or not the diameter can be reduced by reversing at most $k$ arcs. 
\end{proof}

%------------------------------------------------------------
\thmCactusDynamicAlg*

\begin{proof}
    \noindent\textbf{Proof:} Let $ G_0=(V,A)$ be the given orientation of an underlying undirected cactus graph $G=(V,E)$
 and $k_0, \, d_0$ the inputs to the problem. We begin by showing that the problem can be reduced to choosing for each cycle, whether to orient it clockwise or counterclockwise.
    \begin{Claim}
        In any orientation of $G$ with a finite diameter, each of the cycles are oriented as a directed cycle.
    \end{Claim} 
\begin{claimproof}
    To see this, consider an orientation where a cycle $C$ does not have a cyclic orientation. Then, there is a vertex $v\in C$ with no out-neighbor in $C$. If $v$ has no out-neighbor outside of $C$ either, then the orientation has unbounded diameter. Thus, $v$ has some out-neighbor $u\notin C$. Then, let $w \in C$ be a neighbor of $v$. By assumption, the arc $(w,v)$ is oriented from $w$ to $v$. Now observe, that there can be no path from $u$ to $w$ as the cycles are connected acyclically, so the orientation would have an infinite diameter.
\end{claimproof}
 We can then give an arc reversal cost $f$ for each cycle $C$ where $f(C)$ is the number of arc reversals needed to change the orientation of $C$ in $G_0$ to be a clockwise orientation. The number of arc reversals needed for a counterclockwise orientation is then $|C|-f(C)$.
 Let $ C_1, \, C_2,\ldots, C_k $ be the simple cycles of $ G $, and $ T=(\{C_1, \, C_2,\ldots, C_k\}, F) $ 
 the underlying tree where an edge $ \{C_1, \, C_2\}\in F $ whenever $ C_1 $ and $ C_2 $ share a vertex. Fix an arbitrary root vertex $ r $ of $ T $. Given a vertex $ v $ in cycles $ C_1,\ldots, C_q $, let $ C_v $ denote the one furthest from $ r $. 
 In case this is not unique, let us use a slight abuse of notation
and replace $ v $ with a 0 weight cycle of length equal to the number of cycles sharing $v$ without changing the diameter.
	
	We wish to solve the following subproblem for each valid input:	
	Given a vertex $ v $, three positive integers $ d_{to} $, $ d_{from} $, and $k$, output the minimum diameter of an orientation of $ G $ such that the maximum distance from $ v $ to a node of some child of $ C_v $ is at most $ d_{to} $, the maximum distance to $ v $ from a node of some child of $ C_v $ is at most $ d_{from} $, and the orientation can be reached by at most $k$ arc reversals of the edges of $G_0$. We denote this minimum diameter by $ D(v, d_{to}, d_{from}, k) $.
	
	We now describe how to compute $ D(v, d_{to}, d_{from},k) $ from vertices in $ C_v $. Let $ C_v=(v,u_1,\ldots, u_{|C|}) $ be a clockwise traversal of the vertices of $ C_v $, and $v, \, u_{i_1},\ldots,u_{i_\ell}$ the vertices with degree greater than 2 in $G$. For ease of notation, we will refer to the path from $w_{1}$ to $w_{3}$ through $w_{2}$ as $(w_{1}, \, w_{2}, \, w_{3})$, even though it might contain more vertices. 
 We then have to choose between orienting the edges of $ C_v $ clockwise or counterclockwise. 
 Let $\text{dcw}_{to}(u)$ ($\text{dcw}_{from}(u)$) be the distance from $v$ to $u$ (to $v$ from $u$) when $C_v$ is oriented clockwise; let $\text{dccw}_{to}(u)$ ($\text{dccw}_{from}(u)$) be the distance from $v$ to $u$ (to $v$ from $u$) when $C_v$ is oriented counterclockwise. 
 If we orient clockwise, the \emph{to} distances of the children are the following:
	\begin{align*}
		\text{dcw}_{to}(u_{i_1}) &= d_{to}- |(v,u_1,\ldots,u_{i_1})|\\
		\text{dcw}_{to}(u_{i_2}) &= d_{to}- |(v,u_1,\ldots,u_{i_2})|\\
		\vdots\\
		\text{dcw}_{to}(u_{i_\ell}) &= d_{to}- |(v,u_1,\ldots,u_{i_\ell})|.
	\end{align*}
 See Figure \ref{fig:cactusDP} for an example of the reasoning behind these choices.
Similarly, for the maximum distances from a node of one of the children of $ C_v $:
		\begin{align*}
		\text{dcw}_{from}(u_{i_1}) &= d_{from}- |(u_{i_1},u_{i_1+1},\ldots, v)|\\
		\text{dcw}_{from}(u_{i_2}) &= d_{from}- |(u_{i_2},u_{i_2+1},\ldots, v)|\\
		\vdots\\
		\text{dcw}_{from}(u_{i_\ell}) &= d_{from}- |(u_{i_\ell},u_{i_\ell+1},\ldots, v)|.
	\end{align*}

	When orienting counterclockwise, the following parameters are needed:
 
	\begin{align*}
		\text{dccw}_{to}(u_{i_1}) &= d_{to}- |(v, u_{|C|},\ldots,u_{i_1})|\\
		\text{dccw}_{to}(u_{i_2}) &= d_{to}- |(v, u_{|C|},\ldots,u_{i_2})|\\
		\vdots\\
		\text{dccw}_{to}(u_{i_\ell}) &= d_{to}- |(v, u_{|C|},\ldots,u_{i_\ell})|\\
		\text{dccw}_{from}(u_{i_1}) &= d_{from}- |(u_{i_1},u_{i_1-1},\ldots, v)|\\
		\text{dccw}_{from}(u_{i_2}) &= d_{from}- |(u_{i_2},u_{i_2-1},\ldots, v)|\\
		\vdots\\
		\text{dccw}_{from}(u_{i_\ell}) &= d_{from}- |(u_{i_\ell},u_{i_\ell-1},\ldots, v)|.
	\end{align*}

	\begin{figure}
		\centering
		\begin{tikzpicture}
			\draw (0,0) circle[draw=black, radius=1];
			\draw (0.707, -0.707) circle[radius=2pt] node [below right, font=\small] {$ u_1 $};
			\draw (-0.707, -0.707) circle[radius=2pt] node [below left, font=\small] {$ u_2 $};
			\draw (0, 1) circle[radius=2pt] node [above, font=\small] {$ v $};
			\draw (1.414,-1.414) circle[draw=black, radius=1];
			\draw (-1.414,-1.414) circle[draw=black, radius=1];

                \draw[-{Stealth[length=3mm]},color=red,dotted, line width=1.5pt] (0,0.9) arc (90:-45:.9);
                \draw[-{Stealth[length=3mm]},color=red,dotted, line width=1.5pt] (0.7777,-0.7777) arc (135:315:.9);
                
                \draw[-{Stealth[length=3mm]},color=blue, dashed, line width=1.5pt] (0,0.9) arc (90:315:.9);
                \draw[-{Stealth[length=3mm]},color=blue, dashed, line width=1.5pt] (0.8484,-0.8484) arc (135:315:.8);
		\end{tikzpicture}
	\caption{A picture explaining the calculations for orienting $ C_v $. The dotted red path correspond to the calculation of $\text{dcw}_{to}(u_1)$ and the dashed blue correspond to the calculation of $\text{dccw}_{to}(u_1)$.}\label{fig:cactusDP}
	\end{figure}

	Knowing these values, we can calculate the diameter of an orientation of the subgraph rooted at $ v $ as the following maximum. For both orientations, we compute the maximum distance between any two nodes as
	% Fix this
	\begin{align*}
		\text{dcw}_{btw} = \max \big\{ \max_{1\leq i<j\leq k}\{&\text{dcw}_{from}(u_i) + |(u_i,\ldots, u_j)| %w((u_i,u_{i+1})) + \ldots + w((u_{j-1}, u_j)) 
  + \text{dcw}_{to}(u_j)\}, \\
		\max_{1\leq i<j\leq k}\{&\text{dcw}_{from}(u_j) + |(u_j,u_{j+1},\ldots, v,\ldots, u_i)|
  %w((u_j,u_{j+1})) + \ldots + w((u_k,v)) + w((v, u_1)) + \ldots + w((u_{i-1}, u_i)) 
  + \text{dcw}_{from}(u_i)\},\\ 
		\max_{1\leq i\leq k}\{&D(u_i, \text{dcw}_{to}(u_i), \text{dcw}_{from}(u_i), k-f(C_v)\} \big\}\\
		\text{dccw}_{btw} = \max\big\{
		\max_{1\leq i<j\leq k}\{&\text{dccw}_{to}(u_j) + |(u_j,u_{j+1},\ldots, v,\ldots, u_i)|%w((u_j,u_{j+1})) + \ldots + w((u_k,v)) + w((v,to u_1)) + \ldots + w((u_{i-1}, u_i))
  + \text{dccw}_{from}(u_i, )\},\\
		\max_{1\leq i<j\leq k}\{&\text{dcw}_{to}(u_i) + |(u_i,\ldots, u_j)|%w((u_i,u_{i+1})) + \ldots + w((u_{j-1}, u_j))
  + \text{dccw}_{from}(u_j)\},\\
		\max_{1\leq i\leq k}\{&D(u_i, \text{dccw}_{to}(u_i), \text{dccw}_{from}(u_i), k-|C_v|+f(C_v)\} \big\}.
	\end{align*}

	Note that $ d_{to} $ and $ d_{from} $ are bounded by $ \text{dcw}_{btw}$ and $ \text{dccw}_{btw} $ as well, and that if both are finite, each value of $d_{to/from}^{u_i,\text{cw}/\text{ccw}}$ has an orientation achieving it as well, with the appropriate number of arc reversals.
	The orientation we choose is then
	
	\begin{align*}
		D(v, d_{to}, d_{from}, k) = \argmin\{\text{dcw}_{btw}, \text{dccw}_{btw}\}.
	\end{align*}

    Finally, we output the first case where $D(r, d, d, k)$ is less than the input $d_0$, as we increase $k$. If no such entry exists, then we conclude that the problem is infeasible. The actual orientation can then be found via a trace-back procedure.

%	\subsubsection{Correctness}
	Correctness of this algorithm follows from the observation that we are looping through each possible consequence of each possible orientation. The running time can be bounded by the number of possible choices for $v, d_{to}, \, d_{from}, k$ times the time it takes to calculate each entry. 
    Let $n$ be the number of vertices and $m$ be the number edges.
    There are $ nm^3 $ choices of $ v, d_{to}, \, d_{from}, \, k $, as $k_0<m$. Each cycle has at most $n$ nodes,
    so we compute a constant times $n$ values for $ d_{to/from}^{u_i,\text{cw}/\text{ccw}} $, and take the maximum over $O(n^2)$ terms. Thus the algorithm has complexity $O(n^3m^3)$, which is polynomial in the input.

\end{proof}

%------------------------------------------------------------
\subsection{Proofs for Section~\ref{sec:algo_eggeman_noble}}
\label{section_deferred_FPT_k_fixed}

We show that the \textsc{$k$-Reversals} problem can be solved in polynomial time on planar graphs using the following result of Eggeman and Noble.

\begin{theorem}[Thm. 2.5 in \cite{eggemann2009minimizing}]\label{thm:fpt_on_diameter}
    There exists an FPT algorithm parameterized by $d$, the target diameter, that inputs a planar graph $G$ and determines whether it has an orientation with diameter at most $d$.
\end{theorem} 

In their paper, Eggemann and Noble used the fact that planar graphs with bounded diameter also have bounded treewidth, and then they gave an algorithm that is FPT in the treewidth of $G$. The same result holds for planar graphs of bounded \emph{oriented} diameter, as the diameter of an undirected graph is bounded by the diameter of any of its orientations. 

We note that this result naturally extends to the \textsc{$k$-Reversals} problem. The algorithm is very similar in nature to the one given in \cref{section: cactus} so we can factor $k$ in similarly. 
 For this reason, we only give a sketch of the ideas.

In the \textsc{$k$-Reversals} problem, we are given an orientation of a planar graph. We can use the underlying undirected graph, and only use the orientation when factoring in the arc reversal distance.
Let us consider a tree decomposition of this undirected planar graph with minimal treewidth, which we can construct using \cref{thm_treewidth}.
For a given bag in the tree decomposition, we keep track of all the \emph{characteristics} of each vertex in the bag, where a characteristic is a summary of the distances from a given vertex to other vertices in the bag, as well as a list of vertices that are at distance greater than $d$. We can simply add another dimension for the arc reversal distance from the given initial orientation, and increment this appropriately to increase the complexity by a factor of $k$. Thus with only a slight modification, the algorithm of \cite{eggemann2009minimizing} solves the \textsc{$k$-Reversals} problem in time FPT in the diameter on planar graphs as well.

%-------------------------------------------------------------
\subsection{Proofs for Section~\ref{sec:polytopes}}
%-------------------------------------------------------------
\thmEdgePolytope*

We need the following theorem.
\begin{theorem}
    For every integer $i\geq 8$ there exists an integer $j<i$ and a pair of directed graphs $G, \, H$ with diameter $i$ and $j$ respectively such that their respective directed edge polytopes have the same volume.
\end{theorem}

\begin{proof}
    We start by noticing that directed cycles correspond to standard reflexive simplices and that reversing the direction corresponds to a unimodular (and hence in particular volume preserving) transformation of the associated simplex.
    Further, we notice that joining two graphs at a vertex corresponds to taking the free sum of the corresponding simplices, meaning that the volumes of their corresponding directed edge polytopes are multiplied.
    Hence, the volume of the directed edge polytope of a directed cactus graph does not depend on the direction of the individual cycles.

    Now we can construct $G$ and $H$.
    First assume $i$ was even.
    Write $i=2k$.
    For $G$, consider a directed path of length $k$.
    To every edge in that path, attach a path of length $2$ parallel to its ends and direct it the other way so that it now lies in a directed $3$-cycle.
    This results in a directed cactus graph of $3$-cycles with diameter $2k$.
    Reversing exactly one of the cycles that is attached to two other cycles results in the graph $H$ which has diameter $(2k-1)$.

    If $i$ is odd, we can use an analogous construction.
    Write $i=2k+1$.
    For $G$, we start with a directed path of length $k$ again and complete every edge to a $3$-cycle, except the last one which shall be completed to a $4$-cycle instead.
    The construction of $H$ is analogous.
\end{proof}

This proves that the diameter of the graph does not depend on the volume of the polytope.
In fact, $K_{3,3}$ has orientations of diameter $3$, $4$, and $5$, whose associated volumes are $\frac{15}{60}$, $\frac{10}{60}$, and $\frac{13}{60}$ respectively.

\end{appendices}
\end{document}